\documentclass[a4paper,11pt]{article}
 \usepackage[english]{babel}
\usepackage[a4paper]{geometry}
\usepackage{amsmath}
\usepackage{amsthm}
\usepackage{amssymb}
\usepackage{bbm}
\usepackage{color}
\usepackage[dvipsnames]{xcolor}

\usepackage{tikz}
\usetikzlibrary{positioning}
\usetikzlibrary{arrows}
\usepackage{caption}
\usepackage{subcaption}
 
\usepackage{mathrsfs}
\usepackage{enumerate}
 
\usepackage{hyperref} 

\numberwithin{equation}{section}

\def\bR{\mathbb{R}}
\def\bE{\mathbb{E}}
\def\bP{\mathbb{P}}

\def\bN{\mathbb{N}}

\def\cC{\mathcal{C}}

\def\cD{\mathcal{D}}

\def\cF{\mathcal{F}}

\def\cN{\mathcal{N}}

\def\limlaw{\buildrel \cD\over\rightarrow}

\def\tbf{\textbf }

\def\be{\begin{equation}}
\def\ee{\end{equation}}

\def\rhs{r.h.s.\ }
\def\lhs{l.h.s.\ }
\def\eg{e.g.\ }
\def\st{s.t.\ }
\def\wrt{w.r.t.\ }
\def\iid{i.i.d.\ }

\newtheorem{theorem}{Theorem}   

\newtheorem{prop}[theorem]{Proposition}
\newtheorem{lemma}[theorem]{Lemma}

\begin{document}
\title{Spin Covariance Fluctuations in the SK Model\\ at High Temperature}
\author{Christian Brennecke\thanks{Institute for Applied Mathematics, University of Bonn, Endenicher Allee 60, 53115 Bonn, Germany} 
\and Adrien Schertzer\footnotemark[1]
\and  Chen Van Dam\footnotemark[1]}
\maketitle
\begin{abstract}
Based on \cite{H}, it is well known that the rescaled two point correlation functions
        \[  \sqrt{N} \langle \sigma_i ; \sigma_j\rangle  =  \sqrt{N} \big( \langle \sigma_i \sigma_j\rangle -\langle \sigma_i\rangle \langle \sigma_j\rangle\big)  \]
in the Sherrington-Kirkpatrick spin glass model with non-zero external field admit at sufficiently high temperature an explicit non-Gaussian distributional limit as $N\to \infty$. Inspired by recent results from \cite{ABSY, BSXY, BXY}, we provide a novel proof of the distributional convergence which is based on expanding $\langle \sigma_i ; \sigma_j\rangle$ into a sum over suitable weights of self-avoiding paths from vertex $i$ to $j$. Compared to \cite{H}, our key observation is that the path representation of $\langle \sigma_i ; \sigma_j\rangle$ provides a direct explanation of the specific form of the limiting distribution of $\sqrt{N} \langle \sigma_i ; \sigma_j\rangle$. 
\end{abstract}

\section{Introduction} 

Consider $N$ spins $ \sigma_i$, $i\in \{1,\dots,N\}$, with values in $\{-1,1\}$ whose interactions are described by the Sherrington-Kirkpatrick (SK) \cite{SK} Hamiltonian $H_N:\{-1,1\}^N\to\mathbb{R}$ 
		\be \label{eq:defHN} H_N(\sigma) = \sum_{1\leq i < j\leq N} g_{ij} \sigma_{i}\sigma_j+h \sum_{i=1}^{N} \sigma_i.\ee
Here, the $(g_{ij})_{1\leq i<j\leq N}$ are \iid centered Gaussian random variables with variance $t N^{-1}$ for $i\neq j$ and we set $g_{ii} = 0$ for all $i=1,\ldots,N$. The parameter $t\geq 0$ tunes the interaction strength and $h\in\bR$ tunes the external field strength. We assume the $\{g_{ij}\}$ to be realized in some probability space $(\Omega, \cF, \bP)$ and we denote the corresponding expectation by $\mathbb{E}(\cdot)$. We denote the $L^p(\Omega, \cF, \bP)$ norms by $ \| \cdot\|_{L^p(\Omega)} = (\bE\, |\cdot|^p)^{1/p}$. Expectations over other, independent random variables are typically denoted by $ E(\cdot)$.

In this paper, we analyze the distributional behavior of the spin covariances at high temperature in the limit $N\to\infty$. The magnetization $m_i$ of spin $\sigma_i$ and the spin covariance $m_{ij}$ between spins $\sigma_i$ and $\sigma_j$ are defined by 
		\[ m_i   = \langle \sigma_i \rangle \hspace{0.5cm} \text{ and } \hspace{0.5cm} m_{ij}=\langle \sigma_i ; \sigma_j\rangle= \langle \sigma_i \sigma_j\rangle -\langle \sigma_i\rangle \langle \sigma_j\rangle, \]
where $\langle \cdot \rangle$ denotes the Gibbs expectation induced by $H_N$. It is defined so that
		\[   \langle f \rangle = \frac{1}{Z_N} \sum_{\sigma \in \{-1,1\}^N}  f(\sigma) \,e^{H_N(\sigma)}, \hspace{0.5cm} Z_N =    \sum_{\sigma \in \{-1,1\}^N}  e^{H_N(\sigma)} \]
for every observable $ f: \{-1,1\}^N \to \mathbb{R}$. For $t\geq 0$ sufficiently small, it is well known that several quantities of interest related to the SK model have Gaussian fluctuations in the limit $N\to\infty$. This includes for instance the free energy fluctuations \cite{ALR, CN, GT} and the overlap fluctuations \cite[Section 1.10]{Tal1}. A notable exception consists of the two point functions $  m_{ij}$ in case of a non-vanishing external field $h\neq 0$: as shown in \cite{H}, in this case the fluctuations $\sqrt{N} m_{ij}$ converge in distribution to an explicit non-Gaussian limit. 

The goal of this paper is to provide a new proof of the distributional convergence of $\sqrt{N} m_{ij}$. To state our main result, let us denote by $q= q_{t,h}\in [0,1]$ the solution to 
        \[ q = E \tanh^2(h+ \sqrt{t q}Z),  \]
where $ Z\sim \cN(0,1)$ denotes a standard Gaussian random variable and where $E(.)$ denotes the expectation \wrt $Z$ (if $h=0$, we set $q_{t,h}=0$). Moreover, we denote by $\mu = \mu_{t,h}\in [0,1]$ the quantity
        \[ \mu =E \,\text{sech}^4(h+\sqrt{tq}Z).  \]
\begin{theorem}\label{thm:main}  
Assume that $t\geq 0$ is sufficiently small and denote by $ Z_1, Z_2, Z_3$ i.i.d. standard Gaussian random variables. Then, in the limit $N\to \infty$, we have that
        \be \label{eq:Flu} 
        \sqrt{N}m_{ij} \limlaw   \sqrt{\frac{t}{ 1-t\mu}} \, Z_1\,
        \emph{sech}^2\left(h+\sqrt{tq}Z_2\right) \emph{sech}^2\left(h+\sqrt{tq}Z_3\right).  
        \ee    
\end{theorem}

\noindent \textbf{Remarks:}
\begin{enumerate}[1)]
\item The convergence in \eqref{eq:Flu} was proved first in \cite{H}. To be more precise, \cite[Theorem 1.1]{H} shows the convergence of the moments of $ \sqrt{N}m_{ij}$ towards those of the random variable on the \rhs in \eqref{eq:Flu}, which implies \eqref{eq:Flu} by standard arguments. 
\item It is clear that the convergence \eqref{eq:Flu} can only hold true if the parameters $(t,h)$ satisfy the de Almeida-Thouless condition \cite{AT} $ t\mu_{t,h}<1$. Recall that this parameter range corresponds to the expected replica symmetric region of the model. Whether \eqref{eq:Flu} holds true for all $(t,h)$ that satisfy $ t\mu_{t,h}<1$ remains a challenging open question.
\end{enumerate}
\vspace{0.2cm}
As pointed out above, our main result Theorem \ref{thm:main} is not new, but was already proved in \cite{H}. The proof in \cite{H} computes the moments of $\sqrt{N}m_{ij}$ explicitly by relating them to suitable overlap moments and applying interpolation arguments which are based on the cavity method (for more details on this, see \eg \cite[Sections 1.6 to 1.11]{Tal1}). As such, the proof does not, unfortunately, shed much light on how the specific structure of the limiting distribution on the \rhs in \eqref{eq:Flu} emerges. In particular, as pointed out by Talagrand \cite[Research Problem 1.11.2]{Tal1}, one wonders whether the specific form of the limiting distribution arises from some underlying structure. In this paper, we aim to clarify this point by giving a new proof of Theorem \ref{thm:main} which is inspired by recent results in \cite{ABSY, BSXY, BXY} and which is based on the following heuristics: 

Basic mean field theory suggests that the magnetizations $m_i$ satisfy for $t\geq 0$ sufficiently small the cavity equations 
        \be\label{eq:caveq} m_i \approx \tanh\Big( h + \sum_{k\neq i} g_{ik} m_k^{(i)}\Big),  \ee
where $ m_k^{(i)} $ denotes the $k$-th magnetization after spin $i$ has been removed from the system (for more details on our notation, see Section \ref{sec:notation} below). The validity of \eqref{eq:caveq} for small $t\geq 0$ can be proved with a variety of tools and is well understood \cite{Cha,Tal1, ABSY}. By differentiating \eqref{eq:caveq} \wrt the external field, one can derive analogous equations for higher order correlations functions \cite{ABSY}, and in particular that 
        \[\begin{split} 
        m_{ij} &\approx m_{ii}\sum_{k\neq i} g_{ik} m_{kj}^{(i)} =  m_{ii} \,  g_{ij}   m_{jj}^{(i)}  + m_{ii}\sum_{k\neq i,j} g_{ik} m_{kj}^{(i)}. \end{split}\]
Note that the leading order (in $t$) term $ m_{ii}m_{jj}^{(i)}  g_{ij} = (1-m_i^2) ( 1-  (m_{j}^{(i)})^2)g_{ij}  $ suggests that $m_{ij}$ is typically of size $ O(N^{-1/2}) $ and that the scaling limit $ \lim_{N\to\infty}\sqrt{N} m_{ij}$ can not expected to be Gaussian unless $h=0$, in which case $ m_i=m_{j}^{(i)}=0$ by symmetry. 

In order to go a step further and extract the correct candidate for the limiting distribution, we iterate the previous identity and obtain the graphical representation
        \[\begin{split}
         m_{ij} &\approx  m_{ii} \,m_{jj}^{(i)}\Big( g_{ij} + \sum_{k\neq i,j} g_{ik} m_{kk}^{(i,j)} g_{kj} + \!\!\sum_{k,l\neq i,j; k\neq l}\!\!\! g_{ik} m_{kl}^{(i,j)} g_{lj}\Big) \approx m_{ii} \,m_{jj}^{(i)}\sum_{n\geq 0}\sum_{\gamma\in \Gamma_{n+1}^{ij} } w(\gamma). 
         \end{split} \]
Here, $ \Gamma_{n+1}^{ij}$ denotes the set of self-avoiding paths $\gamma$ among $N$ vertices from vertex $i$ to $j$ of length $|\gamma|=n+1$ and the weight $w(\gamma)$ is defined by 
        \be\label{def:wgam}w(\gamma) =  g_{ik_1} m_{k_1k_1}^{(i,j)} g_{k_1k_2}  m_{k_2k_2}^{(i,k_1,j)}g_{k_2k_3}\ldots m_{k_{n}k_{n}}^{(i,k_1,k_2,\ldots,k_{n-1},j)}g_{k_{n} j}, \ee
if $ \gamma =\big( \{i,k_1\}, \{k_1,k_2\},\ldots, \{k_{n},j\} \big)\in \Gamma_{n+1}^{ij} $ and  $n\geq 1$ (for $n=0$ we set $ w(\gamma)= g_{ij}$). 

Now, in order to see how the distribution on the \rhs in \eqref{eq:Flu} emerges from the path representation of $m_{ij}$, let us make two observations. First, based on the cavity equations \eqref{eq:caveq} and the identity $m_{ii}=1-m_i^2$, we expect that
        \[ m_{ii} \approx \text{sech}^2 \Big( h + \sum_{k\neq i} g_{ik} m_k^{(i)}\Big) \stackrel{\cD}{=} \text{sech}^2 \Big(h + \sqrt{t q_N^{(i)}} Z \Big) \approx \text{sech}^2 \big(h + \sqrt{t q} Z \big)  \]
for a Gaussian $ Z = (q_N^{(i)})^{-1/2} \sum_{k\neq i} g_{ik} m_k^{(i)} \sim \cN(0,1)$, which is independent of the disorder $ (g_{kl})_{k,l\neq i}$, and where we approximated in the third step 
        $$q_N^{(i)} =\frac1N\sum_{k\neq i}(m_k^{(i)})^2 \approx q, $$ 
which can also be justified based on \eqref{eq:caveq}. Proceeding similarly for $ m_{jj}^{(i)}$, this suggests
        \be \label{eq:appr1} \sqrt{N} m_{ij} \stackrel{\cD}{\approx} \Big( \sum_{n\geq 0}\sum_{\gamma\in \Gamma_{n+1}^{ij} } \sqrt{N} w(\gamma) \Big)\,\text{sech}^2 \big(h + \sqrt{t q} Z_2 \big)\text{sech}^2 \big(h + \sqrt{t q} Z_3 \big)  \ee
for two i.i.d. Gaussian random variables $ Z_2, Z_3$ that are independent of the first factor on the \rhs in \eqref{eq:appr1} (indeed, as explained below, the correlations between the three factors are negligible as $N\to \infty$). That is, the second and third factors on the \rhs in \eqref{eq:Flu} emerge as a consequence of the vertex weights $ m_{ii}$ and $m_{jj}^{(i)}$, respectively, and they cause the non-Gaussian behavior of the limiting distribution of $\sqrt{N} m_{ij}$ if $h\neq 0$. 

On the other hand, motivated by the leading order term $  \sum_{\gamma\in \Gamma_{1}^{ij} } \sqrt{N} w(\gamma) = \sqrt{N} g_{ij} \stackrel{\cD}{=}\cN(0,t) $ and by the results of \cite{ALR, BSXY}, a few basic moment computations combined with similar arguments as above suggest that the random vector $ \big(T_{n+1}^{ij}\big)_{n\geq 0}$, defined through
        \be \label{def:Tn} T_{n+1}^{ij} = \sum_{\gamma\in \Gamma_{n+1}^{ij} } \sqrt{N} w(\gamma), \ee
is close in distribution to a centered Gaussian vector $ \textbf{Y}= (Y_k)_{k\geq 0}$ with covariance 
        $$ E\, Y_k Y_l = t (t\mu)^k\delta_{kl} = \begin{cases}  t (t\mu)^k &: k=l, \\0 &: k\neq l.  \end{cases}$$ 
By the additivity of the variance for a sum of independent Gaussians, we thus expect  
        \[  \sum_{n\geq 0}\sum_{\gamma\in \Gamma_{n+1}^{ij} } \sqrt{N} w(\gamma) \stackrel{\cD}{\approx} \sum_{n\geq 0} Y_n \stackrel{\cD}{=} \sqrt{ t \sum_{n\geq 0} (t\mu)^n}  Z_1  \]
for some standard normal variable $Z_1$ (independent of $Z_2, Z_3$). With \eqref{eq:appr1}, we arrive at
        \[\sqrt{N} m_{ij} \stackrel{\cD}{\approx} \sqrt{\frac{t}{ 1-t\mu}} \, Z_1\,
       \text{sech}^2 \big(h + \sqrt{t q} Z_2 \big)\text{sech}^2 \big(h + \sqrt{t q} Z_3 \big).  \] 

It turns out that the simple heuristics sketched above can be made rigorous without too much effort, thus providing an alternative and, hopefully, more transparent proof of Theorem \ref{thm:main} compared to \cite{H}. To carry this out, we proceed in two main steps. After introducing basic notation in the next Section \ref{sec:notation}, we first extend the analysis of \cite{ABSY} to derive the following truncated version of the path representation of $ m_{ij}$ in Section \ref{sec:Tbnd}.

    \begin{prop}\label{prop:distT} 
    Let $t\geq 0$ be small enough. Then, there exists a constant  $C>0$, which is independent of $M, N$ and $t$, such that
    \begin{equation*} 
    \Big\| \sqrt{N}   m_{ij} -m_{ii} m_{jj}^{(i)}\sum_{n=0}^{M}T_{n+1}^{ij} \Big\|_2  \leq C  N^{-1/2} + Ct^{M/2} .
    \end{equation*}
    \end{prop}

In the second step, we need to compute the limiting distribution of the random variable $ m_{ii} m_{jj}^{(i)} \sum_{n=0}^{M}T_{n+1}^{ij}$. Inspired by \cite{ALR}, this could for instance be based on a direct moment analysis. This is carried out in \cite{CvD} on which part of this manuscript is based. In this paper instead, we rather find it convenient to follow the above heuristics more closely and we determine the joint limiting distribution of the random vector 
        \be \label{def:XM} \tbf{X}^{(M)}=(X_n)_{n=0}^M = m_{ii} m_{jj}^{(i)}\big(T_{n+1}^{ij}\big)_{n=0}^M\in \bR^{M+1},\ee
for fixed $M\in\bN$. Based on Stein's method, the following is proved in Section \ref{sec:stein}. 

\begin{prop}\label{prop:Tlim} 
    Let $t\geq 0$ be small enough. Then, for every $M\in\bN$, we have in the limit $N\to\infty$ that 
        \[  \tbf{X}^{(M)} \limlaw  \tbf Y^{(M)} \emph{sech}^2 \big(h + \sqrt{t q} Z_2 \big)\emph{sech}^2 \big(h + \sqrt{t q} Z_3 \big) 
         \]
    for a centered Gaussian vector $ \tbf Y^{(M)} = (Y_n)_{n=0}^M$ with covariance $ E\, Y_k Y_l = \delta_{kl }t (t\mu)^k$ and for two i.i.d. standard normal variables $Z_2, Z_3$ which are independent of $  \tbf Y^{(M)}$. 
    \end{prop}

Theorem \ref{thm:main} is now a direct consequence of Prop.\ \ref{prop:distT} and Prop.\ \ref{prop:Tlim}: 
\begin{proof}[Proof of Theorem \ref{thm:main}]
  We assume $t\geq 0$ to be sufficiently small, in particular $t<1$. For $M\in\bN$, let $\tbf X^{(M)}$ and $\tbf Y^{(M)}$ be defined as above and let us denote by $\tbf U^{(M)}$ the vector
        \[ \tbf U^{(M)} =\tbf Y^{(M)} \text{sech}^2 \big(h + \sqrt{t q} Z_2 \big)\text{sech}^2 \big(h + \sqrt{t q} Z_3 \big).\]
We use the standard fact that a sequence $ (\chi_k)_{k\geq 0}$ of $ \bR^d$-valued random variables converges in distribution to $\chi$ if $E f(\chi_k) \to E f(\chi)$ as $k\to \infty$ for all bounded Lipschitz continuous $f:\bR^d\to \bR$ (see \eg \cite[Theorem 11.3.3]{Dud}). Applying this criterion, note that Prop.\ \ref{prop:Tlim} and standard properties of independent Gaussian random variables imply that
        \[ \sum_{n=0}^M X_n \limlaw \sum_{n=0}^M U_n \hspace{0.5cm} \text{as} \hspace{0.5cm} N\to \infty\]
as well as
        \[\begin{split}
              \sum_{n=0}^M U_n &\stackrel{\cD }{=}  \sqrt{t \sum_{n=0}^M(t\mu)^n} \, Z_1\,
       \text{sech}^2 \big(h + \sqrt{t q} Z_2 \big)\text{sech}^2 \big(h + \sqrt{t q} Z_3 \big) \\
&\limlaw \sqrt{\frac{t}{ 1-t\mu}} \, Z_1\,
       \text{sech}^2 \big(h + \sqrt{t q} Z_2 \big)\text{sech}^2 \big(h + \sqrt{t q} Z_3 \big) \equiv V \hspace{0.5cm} \text{as} \hspace{0.5cm} M\to \infty. 
        \end{split}\]
Now, let $f: \bR\to\bR$ be bounded and Lipschitz continuous with Lipschitz constant $L\geq 0$. By Prop.\ \ref{prop:distT}, Cauchy-Schwarz and the previous observations, we find that 
        \[\begin{split}
            \big| \bE f( \sqrt N m_{ij}) - E f(V) \big| \leq &\  L \Big\| \sqrt N m_{ij}- \sum_{n=0}^M X_n\Big\|_2  +   o_N(1) + o_M(1) \\ 
            \leq&\,    L (C N^{-1/2} + Ct^{M/2})  + o_N(1) + o_M(1).
        \end{split} \]
Sending $N\to\infty$ and then $M\to \infty$, we find that $\lim_{N\to\infty}\bE f( \sqrt N m_{ij}) = E f(V)$.
\end{proof}

\section{Notation and Auxiliary Results}\label{sec:notation}

In this section, we introduce some notation and  we recall a few results from \cite{ABSY}. 

First of all, connected correlation functions are defined by
        \[
        m_{i_1 \dots i_n}= \big(\partial_{h_{i_1}} \ldots \partial_{h_{i_n}} \log Z_N\big) |_{h_1, \dots, h_N=h}.
        \]
Below, we also consider expectations of observables conditionally on a given number of spins. To this end, suppose $A = \{j_1, j_2, \dots , j_k\} \subset \{1,\dots, N\}$ is disjoint from $B \subset \{1, \dots , N \}$ with $|B| = l$ and let $\tau = (\tau_{j_1}, \dots , \tau_{j_k}) \in \{-1, 1\}^k$ be a $k$-particle configuration. Then, we define $H_N^{[A,B]} = H_{N,(\tau_{j_1}, \dots , \tau_{j_k})}^{[A,B]}:\{-1,1\}^{N-k-l} \to \bR$ by
        \[
        H_N^{[A,B]}(\sigma)=\sum_{\substack{ 1\leq i < j\leq N \\ i,j \notin A \cup B}} g_{ij} \sigma_{i}\sigma_j+h \sum_{\substack{ 1\leq i \leq N \\ i \notin A \cup B}}^{N} \Big(h+\sum_{j\in A}g_{ij}\tau_j\Big)\sigma_i.
        \]
$H_N^{[A,B]}(\sigma)$ corresponds to the energy of the system, conditionally on $\sigma_j$ for $j\in A$ (\st $\sigma_j=\tau_j$) and with all spins labeled by $j \in B$ being removed from the system. For disjoint subsets $A,B \subset \{1, \dots, N\}$, $ \langle . \rangle^{[A,B]}$ denotes the Gibbs measure induced by $H_N^{[A,B]}$. Notice that $\langle f \rangle^{[A,B]}$ is equal to the conditional expectation of $f$, given the spins $\sigma_j$ for $j \in A$

We abbreviate $\langle . \rangle^{[A]} \equiv \langle . \rangle^{[A, \emptyset]}$, $\langle . \rangle^{(B)} \equiv \langle . \rangle^{[\emptyset, B]}$ as well as $\langle . \rangle \equiv \langle . \rangle^{[\emptyset, \emptyset]}$. By slight abuse of notation, if $A=\{i\}$, we write for simplicity $\langle . \rangle^{[i]}= \langle . \rangle^{[\{i\}]}$ and $\langle . \rangle^{(i)}= \langle . \rangle^{(\{i\})}$.

Finally, for disjoint $A,B \subset \{1, \dots, N\}$, an index $i\in A$ and an observable $f$, we  set 
        \[\begin{split}
        \delta_i \langle f \rangle^{[A,B]}&=\frac{1}{2}\big(\langle f \rangle^{[A,B]}(\sigma_i=1)-\langle f \rangle^{[A,B]}(\sigma_i=-1)\big), \\
        \epsilon_i \langle f \rangle^{[A,B]}&=\frac{1}{2}\big(\langle f \rangle^{[A,B]}(\sigma_i=1)+\langle f \rangle^{[A,B]}(\sigma_i=-1)\big).
        \end{split}\]

An important ingredient in our analysis is that correlation functions are small for $t\geq 0$ small enough. In \cite{ABSY}, this is proved based on the key identities (see \cite[Section 3]{ABSY})
        \begin{equation}\label{eq:ito2pt}
        \begin{split}
        m_{ij}^{[A,B]}&=\Big(1-(m_i^{[A,B]})^2\Big)\delta_i m_j^{[A \cup \{i\},B]}, \\ 
       \text{d} \delta_i m_j^{[A \cup \{i\},B]}&=\sum_{k \notin A\cup B} \epsilon_i m_{kj}^{[A \cup \{i\},B]}\, \text{d}g_{ik}+\frac{1}{2}\sum_{k \notin A\cup B}   \delta_i \left(m_{kkj}^{[A \cup \{i\}, B]}\right) \frac{\text{d}t}{N}, 
        \end{split}
        \end{equation}
and suitable variants thereof (\eg analogous identities for three point functions which follow by differentiation \wrt the external field). The second line is a straightforward consequence of It\^o's lemma, viewing the $i$-th column $ (g_{ik})_{k\neq i} = (g_{ik}(t))_{k\neq i}$ as a rescaled Brownian motion at time $t$. Using basic Gronwall arguments, the following proposition can be proved like [Lemmas 3.1 \& 3.2]\cite{ABSY}; since the proofs are very similar to those of \cite[Lemmas 3.1 \& 3.2]{ABSY} (see also \cite[Lemma 3.1, Remarks 2) \& 3)]{ABSY}), we skip the details. 

\begin{prop}\label{prop:apriori} Let $t\geq 0$ be sufficiently small and let $ A,B\subset \{1,\ldots, N\}$ be disjoint. Then, for some constant $C_t>0$, independent of $N, A$ and $B$, we have that
        \[\begin{split} \sup_{\substack{t_{ij} \in [0,t],\\ 1\leq i<j\leq N }}\sup_{\sigma\in \{-1,1\}^{|A|}} \big\| \delta_i m_{j}^{[A\cup\{i\},B]} \big( (g_{ij}(t_{ij})_{1\leq i<j\leq N })\big)\big\|_4& \leq  C_t\, N^{-(|\{i,j\}|-1)/2},\\
        \sup_{\substack{t_{ij} \in [0,t],\\ 1\leq i<j\leq N }}\sup_{\sigma\in \{-1,1\}^{|A|}} \big\| \delta_i m_{jk}^{[A\cup\{i\},B]} \big( (g_{ij}(t_{ij})_{1\leq i<j\leq N })\big)\big\|_4&\leq C_t\, N^{-(|\{i,j,k\}|-1)/2}
        \end{split}\]
for all $i,j,k\not \in A\cup B$. As a consequence, it follows that
        \[\begin{split} \sup_{\substack{t_{ij} \in [0,t],\\ 1\leq i<j\leq N }}\sup_{\sigma\in \{-1,1\}^{|A|}} \big\|  m_{ij}^{[A,B]} \big( (g_{ij}(t_{ij})_{1\leq i<j\leq N })\big)\big\|_4&\leq  C_t\, N^{-(|\{i,j\}|-1)/2},\\
        \sup_{\substack{t_{ij} \in [0,t],\\ 1\leq i<j\leq N }}\sup_{\sigma\in \{-1,1\}^{|A|}} \big\| m_{ijk}^{[A,B]} \big( (g_{ij}(t_{ij})_{1\leq i<j\leq N })\big)\big\|_2&\leq C_t\, N^{-(|\{i,j,k\}|-1)/2}. 
        \end{split}\]

\end{prop}
\section{Derivation of the Path Representation}\label{sec:Tbnd}
In this section, we derive the path representation of  $\sqrt N m_{ij}$ and we prove Prop.\ \ref{prop:distT}. 

The path representation follows by combining and iterating the key identities in  \eqref{eq:ito2pt}. To make this precise, let us recall the definitions in \eqref{def:wgam}, \eqref{def:Tn} and \eqref{def:XM}, and let us recall that $\Gamma^{ij}_{n+1} $ denotes the set of self-avoiding paths among vertices $\{1,\ldots, N\}$ from $i$ to $j$ of length $ n+1$. Moreover, given
a cavity set $ B\subset \{1,\ldots,N\}$, let us define
        \begin{equation}\label{def:E}
        \begin{split}
        E_{ij}^{(B)}&=m_{ii}\bigg(\sum_{k \not \in B\cup\{i\}}\int_{0}^{t}\big(\epsilon_{i}m_{kj}^{[i,B]}-m_{kj}^{( B\cup \{i\})}\big)\text{d}g_{ik}+\frac{1}{2}\sum_{k\not \in B\cup\{i\}}\int_{0}^{t}\delta_{i}\big(m_{kjj}^{\left[i,B\right]}\big)\left(s\right)\frac{\text{d}s}{N}\bigg), 
        \end{split}
        \end{equation}
as well as  
        \[\begin{split}
        E_{0}&=E_{ij}^{(\emptyset)}, \hspace{0.5cm} E_1=m_{ii}m_{jj}^{(i)}\sum_{k_{1}\neq i,j}g_{ik_{1}}E_{jk_{1}}^{\left(i\right)},\\
        E_{n}&\stackrel{n\geq 2}{=}m_{ii} m_{jj}^{(i)}\sum_{\substack{\gamma =(ik_1\ldots k_nj)\in \Gamma^{ij}_{n+1}}} \frac{ w(\gamma)\, E_{k_{n-1}k_n}^{(i,k_1,\ldots,k_{n-2},j )}  }{m_{k_{n-1} k_{n-1}}^{(i,k_1,\ldots,k_{n-2} )} g_{k_{n-1}k_n} m_{k_{n} k_{n}}^{(i,k_1,\ldots,k_{n-1},j )} } ,\\
         A_{n}&\stackrel{n\geq 2}{=}m_{ii} m_{jj}^{(i)}\sum_{\substack{\gamma =(ik_1\ldots k_nj) \in \Gamma^{ij}_{n+1}}} \frac{ w(\gamma)\, m_{k_{n-1}k_n}^{(i,k_1,\ldots,k_{n-2},j )} }{m_{k_{n-1} k_{n-1}}^{(i,k_1,\ldots,k_{n-2} )} g_{k_{n-1}k_n} m_{k_{n} k_{n}}^{(i,k_1,\ldots,k_{n-1},j )} }. 
        \end{split}\]
Here, we introduced the shorthand notation $\gamma= (ik_1k_2\ldots k_nj)\equiv (\{i,k_1\},\ldots,\{k_n,j\})$. Then, for every $M\geq 2$, a straightforward induction argument shows that  
        \begin{equation}\label{eq:thekey}
        \sqrt{N} m_{ij}=\sum_{n=0}^{M} X_{n} +\sum_{n=0}^{M}\sqrt{N}E_{n}+\sqrt{N}A_{M+1}.
        \end{equation}
Indeed, combining the two identities in \eqref{eq:ito2pt}, we have first of all that 
        \[m_{ij}=m_{ii} m_{jj}^{(i)}g_{ij}+E_0 +m_{ii} \sum_{k_1 \neq i,j} g_{ik_1}m_{k_1j}^{(i)}= N^{-1/2}X_0 + E_0 +m_{ii} \sum_{k_1 \neq i,j} g_{ik_1}m_{k_1j}^{(i)}. \]
Now, we apply the two identities in \eqref{eq:ito2pt} to $m_{k_1j}^{(i)} = m_{jk_1}^{(i)} $, extracting the factor $ m_{jj}^{(i)}$ and expanding through It\^o's lemma in the disorder $ (g_{jk})_{k\neq j}$. This yields 
        \[\begin{split}
        m_{ij}& =N^{-1/2}X_0  +m_{ii}m_{jj}^{\left(i\right)}\sum_{k_{1}\neq i,j}g_{ik_{1}}m_{k_{1}k_{1}}^{\left(i,j\right)}g_{k_{1}j} + E_0+m_{ii} m_{jj}^{(i)}\sum_{k_{1}\neq i,j}g_{ik_{1}}E_{jk_{1}}^{\left(i\right)} \\
        &\hspace{0.5cm} +m_{ii} m_{jj}^{\left(i\right)}\sum_{k_{1}\neq i,j}\sum_{k_2\neq i,j,k_{1}} g_{ik_{1}}m_{k_{1}k_{2}}^{\left(i,j\right)}g_{k_{2}j}\\
        & = N^{-1/2}X_0 +N^{-1/2}X_1 + E_0+E_1+A_{2}.
        \end{split}\]
Iterating this procedure by expanding the two point function $m_{k_{M-1}k_M}^{(i,k_1,\ldots,k_{M-2},j )} $ in the definition $A_M$ \wrt $(g_{k_{M-1}l})_{l\neq i, k_1,\ldots,k_{M-1},j} $, we conclude \eqref{eq:thekey} for every $M\geq 2$.

\begin{proof}[Proof of Prop.\ \ref{prop:distT}]
It is enough to provide suitable bounds on $ \sum_{n=0}^{M}E_{n}$ and $A_{M+1} $. We start to estimate the contributions $ E_n$, for $n\geq 0$. To control $E_0 $, we use that $ 0\leq m_{ii}\leq 1$ and we consider first the martingale term in \eqref{def:E}. Using that 
        \[\Big\Vert \sum_{k\neq i}\int_{0}^{t}\big(\epsilon_{i}m_{kj}^{\left[i\right]}-m_{kj}^{\left(i\right)}\big)\text{d}g_{ik}\Big\Vert _{2}=\bigg( \sum_{k\neq i}\int_{0}^{t}\big\Vert\big(\epsilon_{i}m_{kj}^{\left[i\right]}-m_{kj}^{\left(i\right)}\big)(s)\big\Vert _{2}^2  \frac{\text{d}s}{N}\bigg)^{1/2},  \]
an application of It\^o's lemma \wrt $ (g_{ik})_{k\neq i} $ combined with the triangle inequality,  the basic identity $m_{kjll}^{\left[i\right]}=-2m_{lk}^{\left[i\right]} m_{lj}^{\left[i\right]}-2m_{l}^{\left[i\right]} m_{ljk}^{\left[i\right]}$ and Prop.\ \ref{prop:apriori} implies for $k\neq j$ that 
    \[\begin{split} 
    \big\Vert \big(\epsilon_{i}m_{kj}^{\left[i\right]}-m_{kj}^{\left(i\right)}\big)(s)\big\Vert _{2}
    &=\Big\Vert \sum_{l\ne i}\int_{0}^{s}\delta_{i}m_{kjl}^{\left[i\right]}\text{d}g_{il}+\frac{1}{2}\sum_{l\ne i}\int_{0}^{s}\epsilon_{i}m_{kjll}^{\left[i\right]}\left(u\right)\frac{\text{d} u}{N}\Big\Vert _{2} \leq Cs^{1/2} N^{-1} 
    \end{split}\]
and, similarly, for $k=j$ that $
    \big\Vert \big(\epsilon_{i}m_{kj}^{\left[i\right]}-m_{kj}^{\left(i\right)}\big)(s)\big\Vert _{2} \leq Cs^{1/2} N^{-1/2} $. Hence, we get
        \[ \Big\Vert \sum_{k\neq i}\int_{0}^{t}\big(\epsilon_{i}m_{kj}^{\left[i\right]}-m_{kj}^{\left(i\right)}\big)\text{d}g_{ik}\Big\Vert _{2}\leq CN^{-1}t \]
and arguing similarly for the drift term in \eqref{def:E} yields  $\|E_0\|_2\leq CN^{-1}t$. The same arguments imply $ \|E_{k_1k_2}^{(B)}\|_2\leq CN^{-1}t $ if $k_1\neq k_2$, for $C>0$ independent of $N, t$ and $B$. 

Next, to estimate $ E_n$ for $n\geq 1$, we proceed iteratively. In case of $E_1$, we bound 
        \[\begin{split}
        \Big\| m_{ii} m_{jj}^{(i)}\!\!\sum_{k_{1}\neq i,j}g_{ik_{1}}E_{jk_{1}}^{\left(i\right)}\Big\|^2_2\leq \Big\| \sum_{k_{1}\neq i,j}g_{ik_{1}}E_{jk_{1}}^{\left(i\right)}\Big\|^2_2 = \frac{t}N \sum_{k_{1}\neq i,j}\big\|E_{jk_{1}}^{\left(i\right)}\big\|^2_2 \leq \big(Ct N^{-1}\big)^2,
        \end{split}\]
where, in the second step, we used the independence of the disorder $ (g_{ik})_{k\neq i}$ from $ E_{jk_1}^{(i)}$. Repeating the independence argument and using the fact that $ 0\leq m_{kk}^{[A,B]}\leq 1$ $n$ times in order to control $E_n$ for $n\geq 2$, we arrive at 
        \[\begin{split}
          \big \| E_n \big\|_2^2 & = \Big\| m_{ii} m_{jj}^{(i)}\!\!\!\sum_{\substack{\gamma =(ik_1\ldots k_nj)\\ \in \Gamma^{ij}_{n+1}}}   g_{ik_1} m_{k_1k_1}^{(i,j)} g_{k_1k_2}  m_{k_2k_2}^{(i,k_1,j)}\ldots g_{k_{n-2}k_{n-1}} \, E_{k_{n-1}k_n}^{(i,k_1,\ldots,k_{n-2},j )} g_{k_nj}    \Big\|_2^2 \\
          &\leq  \frac{t}N \sum_{k_1\neq i}  \Big\| m_{k_1k_1}^{(i,j)}\!\! \sum_{\substack{\gamma =(k_1\ldots k_nj) \in \Gamma^{ij}_{n}}}    g_{k_1k_2}  m_{k_2k_2}^{(i,k_1,j)}\ldots g_{k_{n-2}k_{n-1}} \, E_{k_{n-1}k_n}^{(i,k_1,\ldots,k_{n-2},j )} g_{k_nj}    \Big\|_2^2\\
          &\leq   \frac{t}{N} \sum_{k_1\neq i} \frac{t}{N}\sum_{k_2\neq i, k_1} \Big\|\sum_{\substack{\gamma =(k_2\ldots k_nj) \in \Gamma^{ij}_{n-1}}}     m_{k_2k_2}^{(i,k_1,j)}\ldots g_{k_{n-2}k_{n-1}} \, E_{k_{n-1}k_n}^{(i,k_1,\ldots,k_{n-2},j )} g_{k_nj}    \Big\|_2^2\\
          &\leq t^{n} \sup_{k\neq l} \big\| E_{kl}^{(i,k_1,\ldots,k_{n-2},j )}\big\|_2^2 \leq Ct^{n} N^{-2}. 
        \end{split}\]
Similarly, we proceed to control $A_{M+1}$ in which case Prop.\ \ref{prop:apriori} implies
        \[  \| A_{M+1}  \|_2 \leq C\, t^{M/2} \sup_{k\neq l} \big\| m_{kl}^{(i,k_1,\ldots,k_{n-2},j )}\big\|_2 \leq C N^{-1/2}t^{M/2}  \]
and collecting the above bounds, this concludes that 
        \[ 
        \Big\| \sqrt{N}   m_{ij} - \sum_{n=0}^{M} X_n \Big\|_2 \leq \sum_{n=0}^M \sqrt{N}  \| E_n \|_2  + \sqrt{N} \|A_{M+1}\|_2 \leq C N^{-1/2} + C t^{M/2}.\qedhere
        \]
\end{proof}


\section{Limiting Distribution of Path Vector} \label{sec:stein}
Before proving Prop.\ \ref{prop:Tlim}, we begin with some preliminary observations. Given the disorder $\tbf G = (g_{ij})_{1\leq i<j\leq N}$, consider a monomial $g_{k_1 l_1}^{a_1} \dots g_{k_n l_n}^{a_n}$ for mutually distinct pairs $(k_i,l_i)$ and $a_i\in\bN_0$, for $i\in\{1,\ldots,n\}$ and $n\in\bN$. Assuming $ f:\bR^{N(N-1)/2}\to\bR$ to be smooth and bounded with bounded derivatives, the following straightforward estimates are a direct consequence of integration by parts and induction: if all $a_1, \dots, a_n \in 2\bN $, then 
        \begin{equation} \label{eq:ind1}
        \Big|\bE\, g_{k_1 l_1}^{a_1} \dots g_{k_n l_n}^{a_n} f({\bf G}) -  \bE f({\bf G})   \prod_{k=1}^n \Big(\frac tN\Big)^{\frac{a_k}{2}}(a_k-1)!! \Big| \leq C \Big(\frac{t}{N}\Big)^{1+\frac12\sum_{i=1}^{n}a_i}.
        \end{equation}
Otherwise, if $d$ denotes the number of odd integers among $a_1, \ldots, a_n$, we have that
        \begin{equation} \label{eq:ind2}
        \big|\bE \, g_{k_1 l_1}^{a_1} \dots g_{k_n l_n}^{a_n} f({\bf G })  \big| \leq C \Big(\frac{t}{N} \Big)^{ \frac{d}{2}+\frac12\sum_{i=1}^{n}a_i}.
\end{equation}
The constant $C$ in \eqref{eq:ind1} and \eqref{eq:ind2} depends on the norms $\| \partial_\alpha f\|_{\infty}$ for multi-indices with degree $|\alpha|\leq \sum_{i=1}^n a_i$ and on the powers $a_i$, but it is independent of $N$. Below, we apply \eqref{eq:ind1} and \eqref{eq:ind2} in particular to estimate (slight modifications of) contributions of the form $ \bE \, w(\gamma) w(\gamma')$, given two paths $ \gamma, \gamma' \in \Gamma_{n+1}^{ij}$ with $ n+1\leq M+1$ for fixed $M\in\bN$. 
\begin{lemma}\label{lm:m-ind}
Fix $n\in\bN$ and mutually distinct indices $k_l\in \{1,\ldots,N\}$ for $l\in\{1,\dots, n\}$. Moreover, let $ b_l\in\bN $, $ B_l\subset\{1,\ldots, N\}$ and suppose that $\max_{1\leq l\leq n}|b_l|$ and $ \max_{1\leq l\leq n}|B_l| $ are bounded uniformly in $N$. Then, for $t\geq0$ small enough, we have that
    \be\label{eq:weighteval}\bigg|\bE \,\prod_{l=1}^n \left(m_{k_lk_l}^{(B_l)}\right)^{b_l} -\prod_{l=1}^n \, E \,\emph{sech}^{2b_l} \big(h+\sqrt{tq}Z \big) \bigg |\leq CN^{-1/4},
    \ee
where $Z\sim \mathcal{N}(0,1)$ and $C = C_t$ depends on $t$, but is independent of $N$.
\end{lemma}

\begin{proof}
If $h=0$, so that $q=q_{t,h}=0$ for $t$ small, the \lhs in \eqref{eq:weighteval} is equal to zero, so let us assume that $h \neq 0$. In this case, the claim follows by slightly generalizing some of the arguments detailed in the proofs of \cite[Prop.\ 1.2 \& Lemma 4.1]{ABSY}.

We first recall that the creation of a cavity does not affect the local magnetization significantly. Indeed, using It\^o's formula to expand w.r.t.\ $ (g_{jk})_{k=1}^N$, we have 
  \[
  m_i^{[\{j\},B]}-m_i^{(B\cup\{j\})}=\sigma_j \sum_{k \not \in B\cup\{j\}} \int_{0}^{t}m_{ik}^{[\{j\},B]}\text d g_{jk}+ \frac{1}{2}\sum_{k\neq j} \int_{0}^{t} m_{ikk}^{[\{j\},B]}(s)\frac{\text ds}{N}.
  \]
Jensen's inequality and Prop.\ \ref{prop:apriori} imply that
  \begin{equation*} 
      \big \| m_i^{(B)}-m_i^{(B,j)}\big \|_2 =    \big \| \big\langle m_i^{[\{j\},B]}-m_i^{(B,j\})} \big\rangle \big \|_2  \leq \big\langle \big \|  m_i^{[j,B]}-m_i^{(B,j)} \big \|_2 \big\rangle \leq CN^{-1/2}
  \end{equation*}
  for every $i \not\in B\cup\{j\}$. Iterating this argument, we thus conclude that 
        \be \label{eq:auxlm5}\big \| m_i-m_i^{(B)}\big \|_2\leq CN^{-1/2}, \hspace{0.5cm} \big \| q_N-q_N^{(B)}\big \|_2\leq CN^{-1/2}, \ee
where we set for $ B\subset \{1,\ldots,N\}$
        \[ q_N^{(B)} = \frac{1}{N}\sum_{k\not\in B} \big(m_k^{(B)}\big)^2 \hspace{0.5cm} \text{and}\hspace{0.5cm}q_N = q_N^{(\emptyset)}.  \]
Similarly, proceeding as in \cite[Lemma 4.1]{ABSY}, It\^o's formula and Prop.\ \ref{prop:apriori} imply 
        \[\Big \| m_i^{(B)}-\tanh\Big(h+ \sqrt{t q_N^{(B\cup\{i\})}} Z_i^{(B)}\Big)\Big \|_2 \leq C N^{-1/2}, \]
where  
        \be \label{def:ZiB} Z_i^{(B)}= \big(t q_N^{(B\cup\{i\})}\big)^{-1/2} \sum_{j \notin B \cup \{i\}}g_{ij}m_j^{(B\cup\{i\})}\sim \cN(0,1). \ee 

Now, consider indices $ k_l$, powers $b_l $ and cavity sets $ B_l$, for $l\in\{1,\ldots,n\}$, as in the statement of the lemma. Then, a slight variation of the previous observations implies 
       \[\Big \| m_{k_l}^{(B_l)}-\tanh\Big(h+  \sqrt{t q_N^{(B)}} Z^{(B)}_{k_l}\Big)\Big \|_2 \leq C N^{-1/2} \]
for the union $B = \{k_1,\ldots,k_n \} \cup \bigcup_{j=1}^n B_j$ and a basic covariance computation shows   
        \[  \Big(Z^{(B)}_{k_l} \Big)_{l=1}^n \sim \cN\big(0, \textbf{1}_{\bR^n}\big).\]
Combining this with $a^u-b^u=(a-b)\sum_{k=0}^u a^{u-k}b^k$, $0\leq m_{k_lk_l}^{(B_l)}=1-\big(m_{k_l}^{(B_l)}\big)^2\leq 1 $, $ 0\leq \text{sech}(\cdot)\leq 1$, Cauchy-Schwarz, \eqref{eq:auxlm5} and \cite[Prop.\ 1.2]{ABSY}, we thus arrive at
    \[\begin{split}
    &\bigg|\bE \,\prod_{l=1}^n \left(m_{k_lk_l}^{(B_l)}\right)^{b_l} -\prod_{l=1}^n \, \bE \,\text{sech}^{2b_l} \big (h+\sqrt{tq}Z \big) \bigg |\\
    &\leq \bigg|\bE \prod_{l=1}^n\text{sech}^{2b_l} \big (h+\sqrt{tq}Z_{k_l}^{(B)} \big) - \bE\prod_{l=1}^n \, \,\text{sech}^{2b_l} \big (h+\sqrt{tq_{N}^{(B)}} Z_{k_l}^{(B)} \big) \bigg | + CN^{-1/4} \\
    & \leq C \,\bE\, \big| q- q_{N}^{(B)} \big| + CN^{-1/4}\leq CN^{-1/4}. \qedhere
    \end{split}\] 
\end{proof}

Equipped with the previous two lemmas, we are now ready to prove Proposition \ref{prop:Tlim}. 

\begin{proof}[Proof of Prop.\ \ref{prop:Tlim} ] 
Recall the definition of $ \tbf X^{(M)} $ in \eqref{def:XM} and that $\tbf U^{(M)}$ equals
        \[ \tbf U^{(M)} =\tbf Y^{(M)} \text{sech}^2 \big(h + \sqrt{t q} Z_2 \big)\text{sech}^2 \big(h + \sqrt{t q} Z_3 \big),\]
where $ \tbf Y^{(M)} = (Y_n)_{n=0}^M$ is a centered Gaussian vector with covariance $ E\, Y_k Y_l = \delta_{kl }t (t\mu)^k$ and $Z_2, Z_3$ are \iid standard normal independent of $  \tbf Y^{(M)}$. Our goal is to prove that
		\be\label{eq:glsec4} \lim_{N\to\infty}  \bE f  ( \tbf{X}^{(M)}  ) =  E f ( \tbf{U}^{(M)})   \ee
for all $f\in C^\infty_c(\bR^{M+1})$. Notice that the family $ \tbf{U}^{(M)}, (\tbf{X}^{(M)})_{N\geq 2}$ of random variables is tight, because $ \max_{n=0,\ldots, M } \|U_n\|_2\leq C,  \max_{n=0,\ldots, M } \| X_n\|_2\leq C$ for some $ C>0$ that is independent of $N$. Combined with an application of Stone-Weierstrass, \eqref{eq:glsec4} implies $\lim_{N\to\infty}  \bE f ( \tbf{X}^{(M)} ) =  E  f( \tbf{U}^{(M)} )$ for all $f\in C_b(\bR^{M+1})$ so that $ \tbf{X}^{(M)} \limlaw \tbf{U}^{(M)}$ as $N\to\infty$. 

We split the proof of \eqref{eq:glsec4} into two parts and define  
		\[\Sigma^{\frac12} =  \Big(   \delta_{kl} \sqrt{ t (t\mu)^k } \Big)_{  k,l=0}^{M}\in\bR^{M+1\times M+1}.\]
Moreover, denote by $ \tbf W^{(M)} = (W_n)_{n=0}^M \sim \cN(0, \tbf{1}_{\bR^{M+1}})$ a standard Gaussian vector \st
		\[\begin{split} 
		\tbf U^{(M)} &\stackrel{\cD}{=} \Sigma^{\frac12}  \tbf W^{(M)} \text{sech}^2 \big(h + \sqrt{t q} Z_2 \big)\text{sech}^2 \big(h + \sqrt{t q} Z_3 \big) \\
		&\stackrel{\cD}{=} \Sigma^{\frac12}  \tbf W^{(M)} \text{sech}^2 \big(h + \sqrt{t q} Z_i^{(j)} \big)\text{sech}^2 \big(h + \sqrt{t q}  Z_j^{(i)} \big)
		\end{split} \] 
with $ \big( Z_i^{(j)}, Z_j^{(i)}\big)\sim\cN(0,\tbf 1_{\bR^2})$ as defined in \eqref{def:ZiB}. Then, we bound
	\be\label{eq:triangle} \begin{split}
    	 |\bE f  ( \tbf{X}^{(M)}  )-  E f  ( \tbf{U}^{(M)}  )  | &\leq  |\bE E f  ( m_{ii}  m_{jj}^{(i)} \Sigma^{\frac{1}{2}}   \tbf W^{(M)}   )- E f (\tbf{U}^{(M)})|  \\
	 &\hspace{0.5cm}  + |\bE  f  ( \tbf{X}^{(M)}  )- \bE  E f  (m_{ii}  m_{jj}^{(i)} \Sigma^{\frac{1}{2}}   \tbf W^{(M)} )| 
	 \end{split} \ee
and we claim that the two terms on the \rhs in \eqref{eq:triangle} vanish as $N\to \infty$. The first term can be controlled using the same ideas as in the proof of Lemma \ref{lm:m-ind}: we find that
		\[\begin{split}
		& | \bE E f (  \Sigma^{1/2}  \tbf W^{(M)} \text{sech}^2 (h + \sqrt{t q} Z_i^{(j)} )\text{sech}^2 (h + \sqrt{t q}  Z_j^{(i)} )) -\bE  E f ( m_{ii}  m_{jj}^{(i)} \Sigma^{\frac{1}{2}}   \tbf W^{(M)}  )|\\
		&\leq  \| | \nabla f| \|_\infty  \|  |\Sigma^{\frac{1}{2}}\tbf W^{(M)} | \|_2  \|   m_{ii}  m_{jj}^{(i)} - \text{sech}^2( h + \sqrt{tq} Z_i^{(j)}) \text{sech}^2( h + \sqrt{tq} Z_j^{(i)})\|_2 \leq \frac{C} {N^{1/4}}  
		\end{split}\]
for some $C=C_M>0$ that is independent of $N$. Here, the last step uses \cite[Prop.\ 1.2]{ABSY}.

Let us now focus on the second term on the \rhs in \eqref{eq:triangle}. Here, we apply a suitable variant of Stein's method that takes into account the randomness of $ m_{ii}  m_{jj}^{(i)}$ (similar arguments were used recently in \cite{GSS}). For $f\in C^\infty_c(\bR^{M+1})$ and $ 0< \cC\in\bR^{M+1\times M+1}$, set  
	$$V_{f,\cC}( \tbf x) = -\int_0^\infty du\,\Big[ E f\big (  e^{-u} \tbf x+ \sqrt{1-e^{-2u}} \,\cC^{\frac{1}{2}} \tbf{W}^{(M)}\big)  - E  f \big(\cC^{\frac{1}{2}} \tbf{W}^{(M)}\big) \Big]. $$
By the multivariate Stein equation (see \eg \cite[Lemma 2.6]{CGS}), it follows that 
 		\[ f (\tbf x) - E f \big( \cC^{\frac12}\tbf{W}^{(M)}\big) = \text{tr}\,  \big(\cC [\nabla^2 V_{f,\cC}] (\tbf x )\big)- \tbf x\cdot [\nabla V_{f,\cC}]  (\tbf x)
		\]
and an elementary bound shows that $ \|  \partial_\alpha V_{f,\cC} \|_{\infty} \leq |\alpha|^{-1}  \|\partial_\alpha f\|_{\infty}$. Choosing $ \tbf x = \tbf{X}^{(M)}$, $ \cC^{\frac12} = m_{ii}  m_{jj}^{(i)}\Sigma^{\frac{1}{2}} $ and taking the expectation over all the random variables yields 
		\be \label{eq:Stein-tr} \begin{split}
		&\bE f\big(\tbf{X}^{(M)}\big) - E f \big( m_{ii}  m_{jj}^{(i)} \Sigma^{\frac12}\tbf{W}^{(M)}\big) \\
		&= \sum_{n=0}^M  \bE    \big(m_{ii} {m_{jj}^{(i)}}\big)^2 t (t\mu)^{n} \big[\partial_{n}^2 V\big]  \big(\tbf{X}^{(M)}\big) -  \sum_{n=0}^M \bE\, X_n \big[\partial_n V\big] \big(\tbf{X}^{(M)}\big)  ,  
		\end{split}\ee
where we write in the remainder $ V = V_{f,\cC}$ for this choice of $\cC$ and fixed $f\in C^{\infty}_c(\bR^{M+1})$. Prop.\ \ref{prop:Tlim} thus follows if we show that the \rhs in \eqref{eq:Stein-tr} vanishes in the limit $N\to\infty$.
Setting $w^{(i)}(\gamma)= w(\gamma)/g_{ik_1}$ if $ \gamma=(ik_1k_2\ldots k_nj)$, integration by parts in $g_{i\bullet}$ implies
    \[\begin{split}
    &\bE X_n \big[\partial_n V\big] \big(\tbf{X}^{(M)}\big)\\ 
    &=\frac{t}{\sqrt N}\sum_{\substack{\gamma={(ik_1 \dots k_{n}j)}  \in \Gamma_{n+1}^{ij}}} \bE  \, (\partial_{g_{ik_1}} m_{ii}) m_{jj}^{(i)}  w^{(i)}(\gamma)\big[\partial_n V\big] \big(\tbf{X}^{(M)}\big)\\
     &\hspace{0.5cm}+\frac{t}{\sqrt N}\sum_{l=0}^{M}\sum_{\substack{\gamma={(ik_1 \dots k_{n}j)} \in \Gamma_{n+1}^{ij}}}  \bE\,  m_{ii} (\partial_{g_{ik_1}}m_{ii}) (m_{jj}^{(i)})^2 w^{(i)}(\gamma)  \big[\partial_l \partial_n \widetilde{V}\big]  \big(\tbf{X}^{(M)}\big) \\
   &\hspace{0.5cm}+ t \sum_{l=0}^M \sum_{\substack{\gamma =(ik_1\ldots k_nj)\in \Gamma^{ij}_{n+1}, \\ \gamma' =(ik_1 k_2'\ldots k_l'j) \in \Gamma^{ij}_{l+1}}} \bE \big (m_{ii} m_{jj}^{(i)}\big)^2  w^{(i)}(\gamma) w^{(i)}(\gamma')   \big[ \partial_{l}\partial_n V\big] \big(\tbf{X}^{(M)}\big)  \\
    &\hspace{0.5cm}+ t  \sum_{l=0}^M \sum_{\substack{\gamma =(ik_1\ldots k_nj)\in \Gamma^{ij}_{n+1}, \\ \gamma' =(ik_1' k_2'\ldots k_l'j) \in \Gamma^{ij}_{l+1}}}\!\!\!\bE  \, m_{ii} (\partial_{g_{ik_1}}m_{ii}) (m_{jj}^{(i)})^2  w^{(i)}(\gamma) w (\gamma')   \big[\partial_{l}\partial_n V\big] \big(\tbf{X}^{(M)}\big)  \\
 &=\Sigma_{1}^{(n)}+\Sigma_{2}^{(n)}+\Sigma_{3}^{(n)}+\Sigma_{4}^{(n)},
     \end{split}\]
where $\Sigma_{j}^{(n)}$ is defined as the contribution in the $j$-th line on the \rhs and where we set
		\begin{equation*}
		\label{e:F}
		 \big[\partial_l \partial_n \widetilde{V}\big](\tbf x) =  -\int_0^\infty du\,e^{-u } \sqrt{1-e^{-2u}} E  \big(  \Sigma^{\frac{1}{2}}  \tbf{W}^{(M)})_l  [\partial_l \partial_n f]\big(   e^{-u}  \tbf{x} + \sqrt{1-e^{-2u}} \cC^{\frac{1}{2}}\tbf{W}^{(M)}\big).
		\end{equation*}
Notice that $ \| \partial_l \partial_n \widetilde{V}\|_{\infty}\leq C \|\partial_l \partial_n f\|_\infty \int_0^\infty du\, e^{-u} \leq C$ for some $C=C_{f,t,M}$. 

It now turns out that $ \Sigma_{1}^{(n)}, \Sigma_{2}^{(n)}$ and $\Sigma_{4}^{(n)}$ vanish in the limit $N\to\infty$ and that $\Sigma_{3}^{(n)}$ cancels, up to another error that is negligible, with the $n$-th summand in the first term on the \rhs in \eqref{eq:Stein-tr}. To see this, using that $ \|  \partial_n V  \|_{\infty} \leq    \|f\|_{\infty}$, Cauchy-Schwarz implies  
		\[\begin{split}
		|\Sigma_{1}^{(n)}|^2 &\leq  \frac CN\sum_{\substack{\gamma={(ik_1 \dots k_{n}j)}  \in \Gamma_{n+1}^{ij}, \\ \gamma'={(ik_1' \dots k_{n}'j)}  \in \Gamma_{n+1}^{ij} }} \bE   (\partial_{g_{ik_1}} m_{ii}) (\partial_{g_{ik_1'}} m_{ii})    w^{(i)}(\gamma)w^{(i)}(\gamma')\\
		& = \frac CN\sum_{u=0}^{n-1} \sum_{\substack{\gamma={(ik_1 \dots k_{n}j)}  \in \Gamma_{n+1}^{ij}, \\ \gamma'={(ik_1' \dots k_{n}'j)}  \in \Gamma_{n+1}^{ij} : |\gamma\cap\gamma'|=u}} \bE   (\partial_{g_{ik_1}} m_{ii}) (\partial_{g_{ik_1'}} m_{ii})    w^{(i)}(\gamma)w^{(i)}(\gamma')\\
		& + \frac CN \sum_{\gamma={(ik_1 \dots k_{n}j)}  \in \Gamma_{n+1}^{ij}} \bE   (\partial_{g_{ik_1}} m_{ii})^2   w^{(i)}(\gamma)^2.
		\end{split}\] 
		Recalling the explicit form $  w^{(i)}(\gamma) = m_{k_1k_1}^{(i,j)} g_{k_1k_2}  m_{k_2k_2}^{(i,k_1,j)}g_{k_2k_3}\ldots m_{k_{n}k_{n}}^{(i,k_1,k_2,\ldots,k_{n-1},j)}g_{k_{n} j}$, for $\gamma={(ik_1 \dots k_{n}j)}$, the number $ |\gamma\cap\gamma'|$ of joint edges determines the size of each summand according to the bounds \eqref{eq:ind1} and \eqref{eq:ind2}: The last term of the latter is straightforward and of order $O(N^{-1})$. As for the term on the second line, we obtain that for $|\gamma\cap\gamma'| = u $,
		\[\bE   (\partial_{g_{ik_1}} m_{ii}) (\partial_{g_{ik_1'}} m_{ii})    w^{(i)}(\gamma)w^{(i)}(\gamma')\leq C N^{u-2n}.\] 
Since $|\Gamma_{n+1}^{ij}| \leq  N^{n}$ and $ |\{\gamma'\in \Gamma_{n+1}^{ij}: |\gamma\cap\gamma'|=u\}| \leq N^{n-u}$ for fixed $\gamma\in \Gamma_{n+1}^{ij}$, we get
		\[ |\Sigma_{1}^{(n)}|^2\leq  \sum_{u=0}^{n-1} CN^{u-2n-1} \big|  \big\{ \gamma,  \gamma'  \in \Gamma_{n+1}^{ij}: |\gamma\cap\gamma'|=u \big\}\big|+CN^{-1} \leq CN^{-1} \]
 for some $C= C_{n}>0$ that depends on $n\leq M$, but that is independent of $N$. Recalling that $\| \partial_l \partial_n \widetilde{V}\|_{\infty}\leq C$, the same argument shows that $ |\Sigma_{2}^{(n)}|^2\leq CN^{-1}$. 
 
 Next, let us consider $\Sigma_{3}^{(n)}$ and let us compare it with the $n$-th summand in the first term on the \rhs in \eqref{eq:Stein-tr}: then, we have that
 		\[\begin{split}
		&\Sigma_{3}^{(n)} -  \big(m_{ii} {m_{jj}^{(i)}}\big)^2 t (t\mu)^{n} \big[\partial_{n}^2 V\big]  \big(\tbf{X}^{(M)}\big) \\
		& =    t\, \bE \big (m_{ii} m_{jj}^{(i)}\big)^2 \big[ \partial_n^2 V\big] \big(\tbf{X}^{(M)}\big) \bigg(\sum_{\substack{\gamma \in \Gamma^{ij}_{n+1}   }}     (w^{(i)}(\gamma))^2   - (t\mu)^{n} \bigg)  \\
		&\hspace{0.2cm}+ (1-\delta_{n0})\sum_{l=1}^M  \!\!\!\!\!\!\!\sum_{\substack{\gamma =(ik_1\ldots k_nj)\in \Gamma^{ij}_{n+1}, \\ \gamma' =(ik_1 k_2'\ldots k_l'j) \in \Gamma^{ij}_{l+1}: \gamma \neq \gamma'  }} \!\!\!\!\!\!\!\!\!\!\!\!\!\!\!\! t\, \bE \big (m_{ii} m_{jj}^{(i)}\big)^2  w^{(i)}(\gamma) w^{(i)}(\gamma')   \big[ \partial_{l}\partial_n V\big] \big(\tbf{X}^{(M)}\big) = \Sigma_{31}^{(n)}\!+\!\Sigma_{32}^{(n)}.  
		\end{split}\]
The error terms $ \Sigma_{31}^{(n)}$ and $ \Sigma_{32}^{(n)}$ can be controlled through simple second moment estimates as before. Recalling that $ \| \partial_s \partial_t V\|_\infty\leq \frac12 \| \partial_s \partial_tf\|_\infty  $, we have on the one hand that 
		\[\begin{split}
		|\Sigma_{31}^{(n)}|^2 &\leq  C\bE  \bigg( \sum_{\substack{\gamma,\gamma' \in \Gamma^{ij}_{n+1}   }}     (w^{(i)}(\gamma))^2 (w^{(i)}(\gamma'))^2    - 2 (t\mu)^{n} \sum_{\substack{\gamma \in \Gamma^{ij}_{n+1}   }}     (w^{(i)}(\gamma))^2   + (t\mu)^{2n} \bigg). 
		\end{split}\]
Applying \eqref{eq:ind1} \& \eqref{eq:ind2}, that $ |\Gamma_{n+1}^{ij}| = N^{n}+O(N^{n-1})$ and Lemma \ref{lm:m-ind}, we find that
 		\[\begin{split} \bE  \sum_{\substack{\gamma  \in \Gamma^{ij}_{n+1}   }}     (w^{(i)}(\gamma))^2 &=  \bigg[\Big( \frac tN \Big)^{n}  \Big(\big( E \,\text{sech}^{4} \big(h+\sqrt{tq}Z \big) \big)^n+O(N^{-1/4})\Big) + O(N^{-n-1})\bigg] |\Gamma_{n+1}^{ij}| \\
		& = (t\mu)^n + O(N^{-1/4}).
		 \end{split}\]
and, similarly, that
		\[\begin{split}
		&\bE \sum_{\substack{\gamma,\gamma' \in \Gamma^{ij}_{n+1}   }}     (w^{(i)}(\gamma))^2 (w^{(i)}(\gamma'))^2 \\
		 &=     \sum_{\substack{\gamma,\gamma' \in \Gamma^{ij}_{n+1}:\\ \gamma\cap\gamma' =\emptyset   }}  \bE   (w^{(i)}(\gamma))^2 (w^{(i)}(\gamma'))^2 +\sum_{u=1}^{n+1}  \sum_{\substack{\gamma,\gamma' \in \Gamma^{ij}_{n+1}:\\|\gamma\cap\gamma'|=u   }}  \bE   (w^{(i)}(\gamma))^2 (w^{(i)}(\gamma'))^2\\
		& = (t\mu)^{2n}+ O(N^{-1/4}) + \sum_{u=1}^{n-1} O\big( N^{2n-u} \times N^{-2n} \big) = (t\mu)^{2n}+ O(N^{-1/4}).
		\end{split}\]
Here, the scaling $O( N^{2n-u} \times N^{-2n} )$ of the $u$-th summand in the error is obtained by multiplying the number of paths $ |\{\gamma,\gamma'\in \Gamma^{ij}_{n+1}: |\gamma\cap\gamma'|=u\}| = O(N^{2n-u})$ with the leading order decay in $N$ of $\bE  (w^{(i)}(\gamma))^2 (w^{(i)}(\gamma'))^2$: if $ |\gamma\cap\gamma'|=u$, at least $ u-1$ Gaussian edge weights in the product $ (w^{(i)}(\gamma))^2 (w^{(i)}(\gamma'))^2$ have multiplicity four in which case $ 2n-2(u-1)$ edge weights have multiplicity two, leading to an overall $ O( N^{-2(u-1)})\times  O(N^{-2(n-u+1)}) =  O(N^{-2n})$ decay. On the other hand, there can be at most $u$ edge weights having multiplicity four in $ (w^{(i)}(\gamma))^2 (w^{(i)}(\gamma'))^2$ in which case one obtains similarly that $\bE  (w^{(i)}(\gamma))^2 (w^{(i)}(\gamma'))^2 = O( N^{-2u})\times O(N^{-2(n-u)}) =  O(N^{-2n})$. Collecting the previous observations, we conclude that $ |\Sigma_{31}^{(n)}|^2\leq C N^{-1/4}$. 

Finally, let us switch to $\Sigma_{32}^{(n)} $ and $ \Sigma_{4}^{(n)}$; both terms can be controlled with similar arguments as before, so let us only provide a few details for $\Sigma_{32}^{(n)} $. Fixing $ l\in \{1,\ldots, M\}$, it turns out useful to split the $l$-th summand in the definition of $\Sigma_{32}^{(n)} $ into 
		\[\begin{split}  \sum_{u=0}^{\min(n, l)} \sum_{\substack{\gamma =(ik_1\ldots k_nj)\in \Gamma^{ij}_{n+1}, \\ \gamma' =(ik_1 k_2'\ldots k_l'j) \in \Gamma^{ij}_{l+1}: \\ \gamma \neq \gamma', |\gamma\cap\gamma'|=u  }} \!\!\!\!\!\!\!\!\!  \bE \big (m_{ii} m_{jj}^{(i)}\big)^2  w^{(i)}(\gamma) w^{(i)}(\gamma')   \big[ \partial_{l}\partial_n V\big] = \sum_{u=0}^{\min(n, l)} \Sigma_{32u}^{(nl)}  \end{split}\]
and controlling each summand separately; considering for concreteness $  \Sigma_{320}^{(n)}$, for instance, we apply Cauchy-Schwarz and get 
		\[\begin{split}
		|\Sigma_{320}^{(nl)}|^2&\leq   \sum_{\substack{\gamma =(ik_1\ldots k_nj)\in \Gamma^{ij}_{n+1}, \\ \gamma' =(ik_1 k_2'\ldots k_l'j) \in \Gamma^{ij}_{l+1}: \gamma \cap \gamma' =\emptyset }} \sum_{\substack{\tau =(i p_1\ldots p_nj)\in \Gamma^{ij}_{n+1}, \\ \tau' =(ip_1 p_2'\ldots p_l'j) \in \Gamma^{ij}_{l+1}: \tau \cap \tau'=\emptyset  }} \!\!\!\!\!\!\!\!\!  \bE \,  w^{(i)}(\gamma) w^{(i)}(\gamma') w^{(i)}(\tau) w^{(i)}(\tau')\\
		&\leq  \sum_{\substack{1\leq k \leq N: \\k\neq i,j }}\sum_{\substack{\gamma\in\Gamma^{kj}_{n}, \gamma' \in \Gamma^{kj}_{l}: \gamma \cap \gamma' =\emptyset,  \\ \tau\in \Gamma^{kj}_{n}, \tau'  \in \Gamma^{kj}_{l}: \tau \cap \tau' = \emptyset   }} \bE \,  w (\gamma) w (\gamma') w (\tau) w (\tau') \\
		&\hspace{0.5cm} +   \sum_{\substack{1\leq k\neq p\leq N: \\k,p\neq i,j }} \sum_{\substack{\gamma\in\Gamma^{kj}_{n}, \gamma' \in \Gamma^{kj}_{l}: \\ \gamma \cap \gamma' =\emptyset  }} \sum_{\substack{\tau\in \Gamma^{pj}_{n}, \tau'  \in \Gamma^{pj}_{l}:  \\ \tau \cap \tau' =\emptyset  }}  \bE \,  w (\gamma) w (\gamma') w (\tau) w (\tau') .
		\end{split} \]
Now, the bounds \eqref{eq:ind1}, \eqref{eq:ind2}, Lemma \ref{lm:m-ind} and basic combinatorics as above imply that 
		\[\begin{split}
		& \sum_{\substack{1\leq k \leq N: \\k\neq i,j }}\sum_{\substack{\gamma\in\Gamma^{kj}_{n}, \gamma' \in \Gamma^{kj}_{l}: \gamma \cap \gamma' =\emptyset,  \\ \tau\in \Gamma^{kj}_{n}, \tau'  \in \Gamma^{kj}_{l}: \tau \cap \tau' = \emptyset   }} \bE \,  w (\gamma) w (\gamma') w (\tau) w (\tau')\\
		&\leq   \sum_{\substack{1\leq k \leq N: \\k\neq i,j }}\sum_{u=0}^{n+l}  \sum_{\substack{\gamma,\tau \in\Gamma^{kj}_{n}; \gamma',\tau' \in \Gamma^{kj}_{l}:\\ \gamma \cap \gamma' =\emptyset, \tau \cap \tau' = \emptyset, \\ |(\gamma\cup \gamma') \cap (\tau\cup\tau')|=u }} \bE \,  w (\gamma) w (\gamma') w (\tau) w (\tau') \\
		& =  (N-2) \sum_{u=0}^{n+l} O\big( N^{2(n+l-2)-(u-2)}\times N^{u - 2 (n+l)} \big)\leq CN^{-1}.
		\end{split}\]
		The $N^{-(u-2)}$ scaling is due to the fact that is enough to fix $n-1$ edges to have $\gamma=\tau$ (the same for $\gamma',\tau'$). Analogously, one shows that 
		\[\begin{split}
		  \sum_{\substack{1\leq k\neq p\leq N: \\k,p\neq i,j }} \sum_{\substack{\gamma\in\Gamma^{kj}_{n}, \gamma' \in \Gamma^{kj}_{l}: \\ \gamma \cap \gamma' =\emptyset  }} \sum_{\substack{\tau\in \Gamma^{pj}_{n}, \tau'  \in \Gamma^{pj}_{l}:  \\ \tau \cap \tau' =\emptyset  }}  \bE \,  w (\gamma) w (\gamma') w (\tau) w (\tau')\leq CN^{-1}
		\end{split}\]
such that $|\Sigma_{320}^{(nl)}|\leq CN^{-1/2} $. The same arguments imply that $|\Sigma_{32u}^{(nl)}|\leq CN^{-1/2} $ for $u\geq 1$ \st $|\Sigma_{32}^{(n)}|\leq CN^{-1/2} $; finally, as already mentioned, a similar analysis implies $ |\Sigma_{4}^{(n)}|\leq CN^{-1/2}$.

Collecting the above estimates, we have proved that the contribution on \rhs in \eqref{eq:Stein-tr} is of the order $ O(N^{-1/4})$, which concludes the proof.
\end{proof}

\vspace{0.2cm}
\noindent\textbf{Acknowledgements.}  C. B. and A. S. acknowledge support by the Deutsche Forschungsgemeinschaft (DFG, German Research Foundation) under Germany’s Excellence Strategy – GZ 2047/1, Projekt-ID 390685813.


\end{document}